\documentclass{amsart}
\usepackage{amsfonts}
\usepackage{esint}
\usepackage{color}
\usepackage{graphicx}

\newtheorem{thm}{Theorem}[section]
\newtheorem{lemma}[thm]{Lemma}

\newtheorem{corollary}[thm]{Corollary}
\theoremstyle{definition}

\newtheorem{definition}[thm]{Definition}
\newtheorem{example}[thm]{Example}
\theoremstyle{remark}
\newtheorem{remark}[thm]{Remark}
\numberwithin{equation}{section}

\numberwithin{equation}{section}
\newcommand{\R}{\mathbb R}
\newcommand{\N}{\mathbb N}
\newcommand{\C}{\mathcal{C}}
\newcommand{\PP}{\mathcal{P}}
\newcommand{\ve}{\varepsilon}
\newcommand{\lam}{\lambda}
\def \N {\mathbb{N}}
\def \R {\mathbb{R}}

\def \div {\mathrm{div}}

\def \dist {\mathrm{dist}}

\def \suchthat {\ \big | \ }

\def \ve {\varepsilon}

\newcommand{\defeq}{\mathrel{\mathop:}=}

\begin{document}
	
\title[The $\infty$-{F}u{\v{c}}{\'{\i}}k spectrum]{The $\infty$-{F}u{\v{c}}{\'{\i}}k spectrum}

\author[J.V. da Silva, J.D. Rossi and A.M. Salort]{Jo\~{a}o V. da Silva, Julio D. Rossi and Ariel M. Salort}

\address{Departamento de Matem\'atica, FCEyN - Universidad de Buenos Aires and
\hfill\break \indent IMAS - CONICET
\hfill\break \indent Ciudad Universitaria, Pabell\'on I (1428) Av. Cantilo s/n. \hfill\break \indent Buenos Aires, Argentina.}

\email[J.D. Rossi]{jrossi@dm.uba.ar}
\urladdr{http://mate.dm.uba.ar/~jrossi}

\email[A.M. Salort]{asalort@dm.uba.ar}
\urladdr{http://mate.dm.uba.ar/~asalort}

\email[J. V. da Silva]{jdasilva@dm.uba.ar}



\begin{abstract}
In this article we study the behavior as $p \nearrow+\infty$ of the {F}u{\v{c}}ik spectrum for $p$-Laplace operator with zero Dirichlet boundary conditions in a bounded domain $\Omega\subset \R^n$. We characterize the limit equation,
and we provide a description of the limit spectrum. Furthermore, we show some explicit computations of the spectrum for certain configurations of the domain.
\end{abstract}	
\maketitle


\section{Introduction}

Given a bounded smooth domain $\Omega\subset \R^n$, $n\geq 1$, we are interested in studying the asymptotic behavior as $p\to \infty$ of the following non-linear eigenvalue problem
\begin{equation}\label{ecu}
\left\{
\begin{array}{ll}
   -\Delta_p u(x)  =  \alpha_p(u^+)^{p-1}(x)- \beta_p(u^-)^{p-1}(x) & \text{in }  \Omega \\
  u(x)  =  0 & \text{on }  \partial\Omega,
\end{array}
\right.
\end{equation}
where $\Delta_p u\defeq \div(|\nabla u|^{p-2}\nabla u)$ denotes the $p$-Laplace operator and $\alpha_p$ and $\beta_p$ are two real parameters. As usual, $u^\pm = \max\{0,\pm u\}$ mean the positive and negative parts of $u$.  Recall that the set
$$
  \Sigma_p \defeq \{(\alpha_p,\beta_p)\in \R^2 \colon \,\, \mbox{there exists a nontrivial solution } u \mbox{ of }\eqref{ecu} \}
$$
is currently known as the \textit{Fu{\v{c}}{\'{\i}}k spectrum} in honor to the Czech mathematician
Svatopluk Fu{\v{c}}{\'{\i}}k, who in the late '70s, studied this kind of equations  in one space dimension with periodic boundary conditions and their relationship with jumping nonlinearities. More precisely, in \cite{Fucik-libro} it was proved that $\Sigma_2$ for $\Omega = (a,b) \subset \R$ consists in two trivial lines and a family of hyperbolic-like curves passing thought the pairs $(\lam,\lam)$, being $\lam$ an eigenvalue of the (zero) Dirichlet Laplacian in the interval
$(a,b)$. Also, explicit formulas for such curves were found. When regarding the one-dimensional case for $p\neq 2$, the structure of the spectrum is similar, see for instance \cite{Dra-92}.
Throughout the last decades several works have been devoted to studying $\Sigma_p$ in $\R^n$ (for $n\geq 1$). The bibliography on this subject is vast. For the linear case, $p=2$, we refer to the reader the papers \cite{Cu-Go-92, DAN, Da-93, dF-Go-94, Fu-80, Mi-94,S}. When $p\neq 2$ we address, for instance, to references \cite{Cu-dF-Go-99,Cu-dF-Go-98, Pe-04,PS,PS2}.

Observe that problem \eqref{ecu} is closely related with the eigenvalue problem of the (zero) Dirichlet $p$-Laplacian, since, when both parameters $\alpha_p$ and $\beta_p$ are considered to be the same, \eqref{ecu} becomes
\begin{equation} \label{plap}
\left\{
\begin{array}{ll}
  -\Delta_p u(x)  =  \lam |u(x)|^{p-2}u(x) & \text{in }  \Omega \\
  u(x)  =  0 & \text{on }  \partial \Omega
\end{array}
\right.
\end{equation}
and it follows that the pair $(\lam_{k,p},\lam_{k,p})$ belongs to $\Sigma_p$ for each $k\in\N$, where $\lam_{k,p}$ denotes the $k$-th (variational) eigenvalue of \eqref{plap}. It is also straightforward to see that the trivial lines $\{\lam_{1,p}\}\times \R$ and $\R\times \{\lam_{1,p}\}$ belong to $\Sigma_p$. The following facts are well-known in the the literature, see \cite{Cu-Go-92, DAN, Da-93, dF-Go-94, Fu-80, Mi-94} and \cite{Cu-dF-Go-99,Cu-dF-Go-98, Pe-04}:  the trivial lines are isolated in the spectrum and curves in $\Sigma_p$ emanating from each pair $(\lam_{k,p},\lam_{k,p})$ exist locally. Moreover, it is proved that the spectrum contains a continuous non-trivial first curve passing though $(\lam_{2,p},\lam_{2,p})$, which is, in fact asymptotic to the trivial lines, and it admits a variational characterization.

Let us recall some important properties on the spectrum of the $p$-Laplacian. For problem \eqref{plap} there exists a sequence of eigenvalues tending to infinity (note that, in general, it is not known if such a
sequence constitutes the whole spectrum), that is, ($\lambda_{k,p})_{k \geq 1}$ such that there are nontrivial solutions to the problem \eqref{plap}, see \cite{GP1}. It is also known (cf. \cite{Anane}) that the first eigenvalue to \eqref{plap} is isolated, simple and can be variationally characterized as
\begin{equation}\tag{{\bf \text{Eigenv.}}}\label{eqEigen}
   \lambda_{1,p} (\Omega) = \inf_{u \in W^{1,p}_0 (\Omega) \setminus \{0\}} \frac{\| \nabla u\|^p_{L^p (\Omega)}}{\| u\|^p_{L^p (\Omega)}}.
\end{equation}

In the last three decades there was an increasing number of works concerning the study of limit for $p$-Laplacian type problems as $p \to +\infty$. In this direction, one of pioneering works is \cite{BdBM} where it was studied the limit of torsional creep type problems for the $p$-Laplacian, namely
$$
  -\Delta_p u_p(x)  =  1 \quad \text{in} \quad \Omega,
$$
obtaining as ``limit equation'' $|\nabla u| =  1 \quad \text{in} \quad \Omega$ (the well-known \textit{Eikonal equation}) in the viscosity sense. Moreover, $u(x) = \dist(x, \partial \Omega)$ is the corresponding limiting solution (we also recall that more general problems are studied there). On the other hand, regarding the so-called $\infty$-eigenvalue problem, the main reference is \cite{JLM}, where the authors proved that such a quantity is obtained as a limit of the first eigenvalue \eqref{eqEigen} in the following way
$$
	\lam_{1,\infty}(\Omega)=\lim_{p\to\infty} \lam_{1,p}^{1/p}(\Omega).
$$
An interesting piece of information is that such an $\infty$-eigenvalue admits a geometric characterization in terms of the radius of the biggest ball inscribed in $\Omega$:
\begin{equation} \label{lam1}
	\lam_{1,\infty}(\Omega)= \frac{1}{\mathfrak{r}}
\end{equation}
where $\mathfrak{r}(\Omega)=\max\limits_{x\in\Omega} \dist(x,\partial \Omega)$. Moreover, \cite{JLM} also establishes that, up to subsequences, as $p\to\infty$ in \eqref{plap}, uniform limits, $\displaystyle u(x) = \lim_{p \to \infty} u_p(x)$, satisfy the following limit equation
\begin{equation}\label{plap.infty}
\left\{
\begin{array}{ll}
  \min \{-\Delta_\infty u(x), |\nabla u(x)|-\lambda_{1,\infty}(\Omega) u(x)\}  =  0 & \text{in }  \Omega \\
  u(x)  =  0 & \text{on }  \partial \Omega
\end{array}
\right.
\end{equation}
in the viscosity sense, where
$$
   \displaystyle \Delta_\infty u(x) \defeq \sum_{i, j=1}^{N} \frac{\partial u}{\partial x_j}(x)\frac{\partial^2 u}{\partial x_j \partial x_i}(x) \frac{\partial u}{\partial x_i}(x)
$$
is the nowadays well-known \textit{Infinity-Laplacian operator}. Recall that solutions to \eqref{plap.infty} minimize
\begin{equation}\label{eqInfQuot}
  \frac{\|\nabla u\|_{L^{\infty}(\Omega)}}{\|u\|_{L^{\infty}(\Omega)}}
\end{equation}
over all function $W^{1, \infty}_0(\Omega)\setminus \{0\}$. In spite of the fact that the function $u(x) = \dist(x, \partial \Omega)$ minimizes \eqref{eqInfQuot}, it is not always a viscosity solution to \eqref{plap.infty} (cf. \cite{JLM} for more details). Thereafter, in \cite{Ju-Li-05} it is proved that the limit of the second eigenvalue of \eqref{plap} (note that such an eigenvalue is also variational) exists and is obtained as
$$
    \lam_{2,\infty}(\Omega)=\lim_{p\to\infty}\lam_{2,p}^{1/p}(\Omega).
$$
Furthermore, as before, this value also admits a geometric characterization given by
\begin{equation} \label{lam2}
	 \lam_{2,\infty}(\Omega)= \frac{1}{\mathfrak{R}}
\end{equation}
where
$$
   \mathfrak{R}(\Omega) = \sup \left\{r>0\colon \,\, \exists \,\, B_r^1, B_r^2 \subset \Omega \,\,\,\text{such that} \,\,\, B_r^1 \cap B_r^2 = \emptyset\right\}.
$$
In this case, a uniform limit to \eqref{plap} satisfies the following limit equation in the viscosity sense
$$
\left\{
\begin{array}{ll}
  \min\{-\Delta_{\infty}\, u(x), |\nabla u(x)|-  \lam_{2,\infty}(\Omega)u(x)\}  =  0 & \text{in }  \{u>0\} \cap \Omega \\
  \max\{-\Delta_{\infty}\,u(x), -|\nabla u(x)|-\lam_{2,\infty}(\Omega)u(x)\}  =  0 & \text{in }  \{u<0\} \cap \Omega \\
  -\Delta_{\infty}\,u(x) =  0 & \text{in }  \{u=0\} \cap \Omega \\
  u(x)  =  0 & \text{on }  \partial \Omega.
\end{array}
\right.
$$

Concerning limits of higher eigenvalues in \eqref{plap} we also refer to reader the article \cite{Ju-Li-05}. Despite the fact that has been proved in \cite{Ju-Li-05} that the set of such $\infty$-eigenvalues is unbounded, a geometric characterization beyond $\lam_{2,\infty}(\Omega)$ has not been achieved. However, when we bring to light the one-dimensional problem with $\Omega$ being the unit interval $(0,1)$, the spectrum is computed to be the sequence $\{\lam_{k,\infty}\}_{k\in\N}$ given by
\begin{equation} \label{lamk.1d}
\lam_{1,\infty}=k, \qquad \lam_{k,\infty}=2\lam_{1,\infty}, \quad k\in \N.
\end{equation}
For more results concerning the $\infty-$eigenvalue problem we refer to \cite{Champion, Crasta, Hynd, Navarro,Yu}
and references therein.

According to our knowledge, up to date, there is no investigation on the asymptotic behavior of the {F}u{\v{c}}{\'{\i}}k spectrum as $p$ diverges. Therefore, in this manuscript we will turn our attention in studying both the structure and characterization of the \textit{$\infty$-{F}u{\v{c}}{\'{\i}}k spectrum}. Furthermore, in some particular configurations of the domain $\Omega$, we are able to perform explicit computations of the spectrum.

In our first theorem we obtain the equation associated to the $\infty$-{F}u{\v{c}}{\'{\i}}k spectrum, which is obtained
letting $p\to \infty$ in equation \eqref{ecu}.

\begin{thm}\label{teo.eq}
Let $(\alpha_p, \beta_p)_{p> 1} \in \Sigma_p$ be such that $\alpha_p^{1/p}, \beta_p^{1/p}$ are bounded and $u_p \in W^{1, p}_0(\Omega)$ a corresponding eigenfunction normalized with $\|u_p\|_{L^p(\Omega)} = 1$. Then, up to a subsequence,
$$
(\alpha_{\infty}, \beta_{\infty}) = \lim_{p \to \infty} \left(\alpha_p^{1/p}, \beta_p^{1/p}\right) \qquad \text{ and } \qquad \lim_{p\to\infty} u_p(x) = u_\infty(x) \text{ uniformly in }\Omega.
$$
Moreover, the limit $u_{\infty}$ belongs to $W^{1,\infty}_0(\Omega)$ and is a viscosity solution to
\begin{equation} \label{ecu.infty}
\left\{
\begin{array}{ll}
  \min\{-\Delta_{\infty}\,u_{\infty}(x), |\nabla u_{\infty}(x)|-\alpha_{\infty} u_{\infty}^{+}(x)\}  =  0 & \text{in }  \{u_{\infty}>0\} \cap \Omega \\
  \max\{-\Delta_{\infty}\,u_{\infty}(x), -|\nabla u_{\infty}(x)|+\beta_{\infty} u_{\infty}^{-}(x)\}  =  0 & \text{in }  \{u_{\infty}<0\} \cap \Omega \\
  -\Delta_{\infty}\,u_{\infty}(x) =  0 & \text{in }  \{u_{\infty}=0\} \cap \Omega \\
  u_{\infty}(x)  =  0 & \text{on }  \partial \Omega.
\end{array}
\right.
\end{equation}
\end{thm}

Regarding the limit equation, we define the $\infty-${F}u{\v{c}}{\'{\i}}k spectrum as
$$
  \Sigma_\infty \defeq \Big\{(\alpha,\beta)\in \R^2 \colon \,\, \mbox{there exists a nontrivial viscosity solution } u \mbox{ of }\eqref{ecu.infty} \Big\},
$$
such a $u$ is defined to be an eigenfunction of the pair $(\alpha,\beta)$.
Observe that, by construction,  eigenfunction of \eqref{ecu.infty} belong to $W^{1,\infty}_0(\Omega)$.

When $\Omega$ is the unit interval in $\R$, a full characterization of the limit of  the $p$-{F}u{\v{c}}{\'{\i}}k spectrum is obtained.
\begin{thm} \label{teo.1d}
The limit of the spectrum $\Sigma_p$ as $p\to \infty$ when $\Omega$ is the unit interval $(0,1)\subset \R$ is given by $$ \displaystyle \Sigma_\infty = \bigcup\limits_{k=1}^{\infty} \C_{k,\infty}^\pm, $$ where
\begin{align*}
&\C_{k,\infty}=\left\{ (k(1+s^{-1}),k(1+s)), \, s \in \R^+ \right\} &\quad \text{ if }k \text{ is even}\\
&\C_{k,\infty}^{+}=\left\{ (k-1+s^{-1}(k+1), k+1+s(k-1)), \, s \in \R^+ \right\} &\quad \text{ if }k \text{ is odd}\\
&\C_{k,\infty}^{-}=\left\{ (k+1+ s^{-1}(k-1), k-1+s(k+1)), \, s \in \R^+ \right\} &\quad \text{ if }k \text{ is odd}.
\end{align*}
\end{thm}

In the higher dimensional case the first nontrivial curve in the $\infty$-{F}u{\v{c}}{\'{\i}}k spectrum can be characterized as follows.
\begin{thm} \label{teo.rn}
The trivial lines in the spectrum of \eqref{ecu.infty} are given by
$$
	\C_{1,\infty}^{+}  = \R\times \left\{\frac{1}{\mathfrak{r}}\right\} \quad \text{ and } \quad \C_{1,\infty}^{-}= \left\{\frac{1}{\mathfrak{r}}\right\} \times \R,
$$
where $\mathfrak{r}=\mathfrak{r}(\Omega)$ is the radius of the biggest ball inscribed in $\Omega$. Moreover, the first non-trivial curve in $\Sigma_\infty$ is parametrized as
\begin{equation} \label{curva.2}
	\C_{2,\infty} = \{ (\alpha_\infty(t),\beta_\infty(t)), \,\,\,t\in \R^+\}
\end{equation}
where $\alpha_\infty(t)=t^{-1} c_\infty(t)$ and $\beta_\infty(t)=c_\infty(t)$, and
$$
   c_\infty(t)=\inf_{\PP_2(\Omega)} \max\left\{\frac{t}{\mathfrak{r}(\omega_1)}, \frac{1}{\mathfrak{r}(\omega_2)}\right\}, \,\,\, t\in \R^+.
$$
Here $\PP_2(\Omega)$ denotes the class of all partitions in two disjoint and connected subsets of $\Omega$. Given $(\omega_1, \omega_2)\in \PP_2(\Omega)$, we denote $\mathfrak{r}_i=\mathfrak{r}_i(\omega_i)$ the radius of the biggest ball inscribed in $\omega_i$ ($i=1,2$).

Moreover, the trivial curves $\C_{1,\infty}^{+}$ and $\C_{1,\infty}^{-}$ intersect the second curve $\C_{2,\infty} $
for almost any domain. In fact, the only exception where $\C_{2,\infty} $ is asymptotic to $\C_{1,\infty}^{+}$ and $\C_{1,\infty}^{-}$ is the ball, $\Omega = B_R$.
\end{thm}

It is remarkable that, contrary to expectations, Theorem \ref{teo.rn} shows that  $\C_{2,\infty}$ does not behaves as a hyperbolic curve asymptotic to the trivial lines, in fact, in most cases (the only exception is the ball) it intersect them (compare with \cite{Cu-dF-Go-99, Cu-dF-Go-98, Dra-92, Pe-04} for the $p$-Laplacian counterpart). In Section \ref{sec4}, we will present a complete description about the behaviour of $\C_{2,\infty}$ (see in particular Corollary \ref{cor4.1}). Finally, in Section \ref{ClasFucikSpect}, we will present some interesting examples in order to illustrate such an unusual phenomena for $\C_{2,\infty}$.

In conclusion, we would   highlight that our approach is flexible enough in order to be applied for other classes of degenerate operators with $p$-Laplacian structure. Some enlightening examples are

 \begin{itemize}

   \item[\checkmark] Anisotropic operators like the \textit{pseudo $p$-Laplacian}
$$
   \displaystyle \tilde{\Delta}_p u \defeq \sum_{i=1}^{N} \frac{\partial }{\partial x_i}\left(\left|\frac{\partial u}{\partial x_i}\right|^{p-2}\frac{\partial u}{\partial x_i}\right).
$$
The eigenvalue problem and its corresponding limit as $p \to \infty$ for such a class of operators were studied in \cite{BK}.

   \item[\checkmark] Nonlocal operators like the \textit{fractional $p$-Laplacian}
$$
    \left(-\Delta\right)_p^s u(x) \defeq \displaystyle C_{N, s, p}.\text{P.V.}\int_{\R^N} \frac{|u(y)-u(x)|^{p-2}(u(y)-u(x))}{|x-y|^{n+sp}}dy.
$$
where $p>1$, $s \in (0, 1)$, P.V. stands for the Cauchy principal value and $C_{N, s, p}$ is a normalizing constant. The mathematical tools in order to study the $\infty$-Fu{\v{c}}{\'{\i}}k spectrum for this class of operators can be found in the following articles,  \cite{FP}, \cite{BKGS}, \cite{LL} and \cite{PSY}.
\end{itemize}

The manuscript is organized as follows: In Section \ref{Prel} we introduce the mathematical machinery (notation and definitions) and several auxiliary results which play an important role in order to prove
Theorem \ref{teo.eq}. In Section \ref{LimProb} we prove Theorem \ref{teo.eq}. In Subsection \ref{sect-1-d} we study in detail the one-dimensional case. The general case is analyzed in Subsection \ref{GenCase}. Finally, Section \ref{ClasFucikSpect}
 is devoted to present several examples where explicit computation of the spectrum are made, as well as the profile of such solutions.

\section{Preliminary results}\label{Prel}

In this section we introduce some definitions and auxiliary lemmas we will use throughout this article.

Let us start by defining the notion of weak solution to
\begin{equation} \label{plap.g}
 - \Delta_p u = g(u) \quad \text{in} \quad \Omega,
\end{equation}
 where $g: \R \to \R$ is a continuous function. Since we will study the asymptotic behavior as $p \to \infty $, without loss of generality we can assume that $p>\max\{2,n\}$.

\begin{definition}\label{DefWS}
  A function $u \in W^{1,p}(\Omega) \cap C(\Omega)$ is said to be a weak solution to \eqref{plap.g} if it fulfills
  $$
  \displaystyle \int_{\Omega} |\nabla u|^{p-2}\nabla u \cdot \nabla \phi\, dx = \int_{\Omega} g(u)\phi \,dx, \qquad \forall \phi \in C^{\infty}_0(\Omega).
  $$
\end{definition}
Since $p > 2$, \eqref{plap.g} is not singular at points where the gradient vanishes, and  consequently, the mapping
$$
   x \mapsto \Delta_p \phi(x) = |\nabla \phi(x)|^{p-2}\Delta \phi(x) + (p-2)|\nabla \phi(x)|^{p-4}\Delta_{\infty} \phi(x)
$$
is well-defined and continuous for all $\phi \in C^2(\Omega)$.

Next, we give the definition of viscosity solutions to \eqref{plap.g}. For the reader's convenience we recommend the survey \cite{CIL} on theory of viscosity solutions.

\begin{definition}
  An upper (resp. lower) semi-continuous function $u: \Omega \to \R$ is said to be a viscosity sub-solution (resp. super-solution) to \eqref{plap.g} if, whenever $x_0 \in \Omega$ and $\phi \in C^2(\Omega)$ are such that $u-\phi$ has a strict local maximum (resp. minimum) at $x_0$, then
$$
    -\Delta_p \phi(x_0) \geq g(\phi(x_0)) \quad (\text{resp.} \,\,\,\leq g(\phi(x_0))).
$$
Finally, a function $u \in C(\Omega)$ is said to be a viscosity solution to \eqref{plap.g} if it is
simultaneously a viscosity sub-solution and a viscosity super-solution.
\end{definition}

Throughout this article, we will consider $g$ defined as
$$
   g(u(x)) = \alpha_p(u^{+})^{p-1}(x)- \beta_p(u^{-})^{p-1}(x).
$$

The following lemmas will be useful for our arguments.

\begin{lemma}\label{Lemma2.4} Assume $n<p < \infty$ and let $u \in W^{1, p}_0(\Omega)$ be a weak solution to \eqref{ecu} normalized by $\|u\|_{L^p(\Omega)} = 1$. Then, $u \in C^{0, \gamma}(\Omega)$, where $\gamma = 1- \frac{n}{p}$. Moreover, the following holds
\begin{itemize}
  \item[\checkmark] $L^{\infty}$-bounds
  $$
  \|u\|_{L^{\infty}(\Omega)} \leq \mathfrak{C}_1,
  $$
  \item[\checkmark] H\"{o}lder estimate
  $$
  \frac{|u(x)-u(y)|}{|x-y|^{\gamma}} \leq \mathfrak{C}_2,
  $$
 \end{itemize}
  where $\mathfrak{C}_1$ and $\mathfrak{C}_2$ are constants depending on $n$, $\alpha_p$ and $\beta_p$.
\end{lemma}

\begin{proof}
Multiplying \eqref{ecu} by $u$ and integrating by parts we obtain
$$
  \int_{\Omega} |\nabla u|^p \, dx  =   \alpha_p\int_{\Omega} u_{+}^{p} \, dx + \beta_p\int_{\Omega} u_{-}^{p}\, dx  \leq   \max\{\alpha_p, \beta_p\}.
$$
Now, by Morrey's estimates and the previous inequality, there is a  positive constant $\mathfrak{C}=\mathfrak{C}(n,\Omega)$  independent on $p$ such that
$$
   \|u\|_{L^{\infty}(\Omega)} \leq \mathfrak{C} \|\nabla u\|_{L^p(\Omega)} \leq \mathfrak{C}\max\left\{\alpha_p^{1/p}, \beta_p^{1/p} \right\},
$$
which proves the first statement.

For the second part, since $p > n$, combining the H\"{o}lder's inequality and Morrey's estimates we have
$$
   \frac{|u(x)-u(y)|}{|x-y|^{\gamma}} \leq \mathfrak{C} \|\nabla u\|_{L^n(\Omega)} \leq \mathfrak{C}|\Omega|^{\frac{p-n}{pn}}\|\nabla u\|_{L^p(\Omega)}\leq \hat{\mathfrak{C}}|\Omega|^{\frac{p-n}{pn}},
$$
where $\hat{\mathfrak{C}}$ depends only on $n$ and $\Omega$.
\end{proof}

The last result implies that any family of weak solutions to \eqref{ecu} with $\alpha_p^{1/p}$, $\beta_p^{1/p}$ bounded is pre-compact in the uniform topology. Therefore, the existence of a uniform limit is established in Theorem \ref{teo.eq}.

\begin{lemma} Let $\{u_p\}_{p>1}$ be a sequence of weak solutions to \eqref{ecu}. Suppose that $\left(\alpha_p^{1/p}, \beta_p^{1/p}\right) \to (\alpha_{\infty}, \beta_{\infty})$ as $p \to \infty$. Then, there exists a subsequence $p_i \to \infty$ and a limit function $u_{\infty}$ such that
$$
   \displaystyle \lim_{p_i \to \infty} u_{p_i}(x) = u_{\infty}(x)
$$
uniformly in $\Omega$. Moreover, $u_{\infty}$ is Lipschitz continuous with
$$
\frac{|u_{\infty}(x)-u_{\infty}(y)|}{|x-y|} \leq \mathfrak{C}\max\left\{\alpha_{\infty}, \beta_{\infty}\right\}.
$$
\end{lemma}
\begin{proof} Existence of $u_{\infty}$ as a uniform limit is a direct consequence of the Lemma \ref{Lemma2.4} combined with an Arzel\`{a}-Ascoli compactness criteria. Finally, the last statement holds by passing to the limit in the H\"{o}lder estimates from Lemma \ref{Lemma2.4}.
\end{proof}

The following lemma gives a relation between weak and viscosity sub and super-solution to \eqref{plap.g}.

\begin{lemma}\label{EquivSols} A continuous weak sub-solution (resp. super-solution) $u \in W_{\text{loc}}^{1,p}(\Omega)$ to \eqref{plap.g} is a viscosity sub-solution (resp. super-solution) to
$$
   -\left[|\nabla u|^{p-2} \Delta u + (p-2)|\nabla u(x)|^{p-4}\Delta_{\infty} u\right] = g(u(x)) \quad \text{in} \quad \Omega.
$$
\end{lemma}

\begin{proof} Let us proceed for the case of super-solutions. Fix $x_0 \in \Omega$ and $\phi \in C^2(\Omega)$ such that $\phi$ touches $u$ by bellow at $x=x_0$, i.e., $u(x_0) = \phi(x_0)$ and $u(x)> \phi(x)$ for $x \neq x_0$. Our goal is to establish that
$$
  -\left[|\nabla \phi(x_0)|^{p-2}\Delta \phi(x_0) + (p-2)|\nabla \phi(x_0)|^{p-4}\Delta_{\infty} \phi(x_0)\right] -g(\phi(x_0)) \geq 0.
$$
Let us suppose, for sake of contradiction, that the inequality does not hold. Then, by continuity there exists  $r>0$ small enough such that
$$
   -\left[|\nabla \phi(x)|^{p-2}\Delta \phi(x) + (p-2)|\nabla \phi(x)|^{p-4}\Delta_{\infty} \phi(x)\right] -g(\phi(x)) < 0,
$$
provided that $x \in B_r(x_0)$. Now, we define the function
$$
   \Psi \defeq \phi+ \frac{1}{10}\mathfrak{m}, \quad \text{ where } \quad \mathfrak{m} \defeq \inf_{\partial B_r(x_0)} (u(x)-\phi(x)).
$$
Notice that $\Psi$ verifies $\Psi < u$ on $\partial B_r(x_0)$, $\Psi(x_0)> u(x_0)$ and
\begin{equation}\label{EqPsi}
 -\Delta_p \Psi(x) < g(\phi(x)).
\end{equation}
By extending by zero outside  $B_r(x_0)$, we may use $(\Psi-u)_{+}$ as a test function in \eqref{plap.g}. Moreover, since $u$ is a weak super-solution, we obtain
\begin{equation}\label{Eq3.4}
  \displaystyle \int_{\{\Psi>u\}} |\nabla u|^{p-2}\nabla u \cdot \nabla (\Psi-u) dx \geq  \int_{\{\Psi>u\}} g(u)(\Psi-u) dx.
\end{equation}
On the other hand, multiplying \eqref{EqPsi} by $\Psi- u$ and integrating by parts we get
\begin{equation}\label{Eq3.5}
  \displaystyle \int_{\{\Psi>u\}} |\nabla \Psi|^{p-2}\nabla \Psi \cdot \nabla (\Psi-u) dx <  \int_{\{\psi>u\}} g(\phi)(\Psi-u) dx.
\end{equation}
Next, subtracting \eqref{Eq3.5} from  \eqref{Eq3.4} we obtain
\begin{align} \label{exxx}
  \int_{\{\Psi>u\}} (|\nabla \Psi|^{p-2}\nabla \Psi - |\nabla u|^{p-2}\nabla u) \cdot \nabla (\Psi-u) dx <  \int_{\{\psi>u\}} \mathcal{G}_{\phi}(u)(\Psi-u)dx,
\end{align}
where we have denoted $\mathcal{G}_{\phi}(u)=g(\phi)-g(u)$.
Finally, since the left hand side in \eqref{exxx} is bounded by below by
$$
  \mathfrak{C}(p)\int_{\{\Psi>u\}} |\nabla \Psi- \nabla u|^pdx,
$$
and the right hand side in \eqref{exxx} is negative, we can conclude that $\Psi \leq u$ in $B_r(x_0)$. However, this contradicts the fact that $\Psi(x_0)>u(x_0)$. Such a contradiction proves that $u$ is a viscosity super-solution.

An analogous argument can be applied to treat the sub-solution case.
\end{proof}

\section{The limiting problem: Proof of Theorem \ref{teo.eq}}\label{LimProb}

In this section we deal with the limit equation obtained as $p\to\infty$ in \eqref{ecu}. We prove that, as $p\to\infty$, weak solutions of \eqref{ecu} converge uniformly to a limit function which, in fact, is characterized to satisfy \eqref{ecu.infty} in the viscosity sense.

\begin{proof}[Proof of Theorem \ref{teo.eq}] First of all, we prove that the limiting function $u_{\infty}$ is $\infty$-harmonic in its null set, i.e.,
$$
  - \Delta_{\infty} u_{\infty}(x) = 0 \quad \text{in} \quad \{u_{\infty} = 0\} \cap \Omega.
$$
To this end, let $x_0 \in \{u_{\infty} = 0\} \cap \Omega$ and $\phi \in C^2(\Omega)$ such that $u_{\infty}-\phi$ has a strict local maximum (resp. strict local minimum) at $x_0$. Since, up to subsequence,  $u_p \to u_{\infty}$ local uniformly, there exists a sequence $x_p \to x_0$ such that $u_p-\phi$ has a local maximum (resp. local minimum) at $x_p$. Moreover, if $u_p$ is a weak solution (consequently a viscosity solution according to Lemma \ref{EquivSols}) to \eqref{ecu} we obtain
$$
  -\left[|\nabla \phi(x_p)|^{p-2}\Delta \phi(x_p) + (p-2)|\nabla \phi(x_p)|^{p-4}\Delta_{\infty} \phi(x_p)\right] \leq g(u(x_p)) \quad (\text{resp.}\,\, \geq ).
$$
Now, if $|\nabla \phi(x_0)| \neq 0$ we may divide both sides of the above inequality by $(p-2)|\nabla \phi(x_p)|^{p-4}$ (which is different from zero for $p$ large enough). Thus, we obtain that
$$
 - \Delta_{\infty} \phi(x_p) \leq \frac{|\nabla \phi(x_p)|^2 \Delta \phi(x_p)}{p-2} + \frac{g( u(x_p))}{(p-2)|\nabla \phi(x_p)|^{p-4}} \quad (\text{resp.}\,\, \geq ),
$$
where the RHS tends to zero as $p \to \infty$, because $g(u(x_p)) \to g(u(x_0)) = 0$. Therefore,
$$
    - \Delta_{\infty} \phi(x_0) \leq 0 \quad (\text{resp.}\,\, \geq 0),
$$
and since such an inequality is immediately satisfied if $|\nabla \phi(x_0)| = 0$ we conclude that $u_{\infty}$ is a viscosity sub-solution (resp. super-solution) to the desired equation.

Next, we will prove that $u_{\infty}$ is a viscosity solution to
$$
  \max\{- \Delta_{\infty} u_{\infty}(x), -|\nabla u_{\infty}(x)|+\beta_{\infty} u^{-}_{\infty}(x)\} = 0 \quad \text{in} \quad \{u_{\infty}<0\} \cap \Omega.
$$
First let us prove that $u_{\infty}$ is a viscosity super-solution.
Fix $x_0 \in \{u_{\infty}<0\} \cap \Omega$ and let  $\phi \in C^2(\Omega)$ be a test function  such that $u_{\infty}(x_0) = \phi(x_0)$ and the inequality $u_{\infty}(x) > \phi(x)$ holds for all $x \neq x_0$.
We want to show that
$$
  - \Delta_{\infty} \phi(x_0) \geq 0 \quad \text{or} \quad -|\nabla \phi(x_0)|+\beta_{\infty} \phi^{-}(x_0) \geq 0.
$$
Notice that if $|\nabla \phi(x_0)| = 0$ there is nothing to prove. Hence, as a matter of fact, we may assume that
\begin{equation}\label{eq5.1}
  -|\nabla \phi(x_0)|+\beta_{\infty} \phi^{-}(x_0)<0.
\end{equation}
As in the previous case, there exists a sequence $x_p \to x_0$ such that $u_p-\phi$ has a local minimum at $x_p$. Since $u_p$ is a weak super-solution (consequently a viscosity super-solution according to Lemma \ref{EquivSols}) to \eqref{ecu} we get
$$
  -\left[|\nabla \phi(x_p)|^{p-2}\Delta \phi(x_p) + (p-2)|\nabla \phi(x_p)|^{p-4}\Delta_{\infty} \phi(x_p)\right] \geq -\beta_p(u_p^{-}(x_p))^{p-1}.
$$
Now, dividing both sides by $(p-2)|\nabla \phi(x_p)|^{p-4}$ (which is different from zero for $p$ large enough due to \eqref{eq5.1}) we get
$$
  - \Delta_{\infty} \phi(x_p) \geq - \frac{|\nabla \phi(x_p)|^2 \Delta \phi(x_p)}{p-2} -\left( \frac{\beta_p^{\frac{1}{p-4}}u_p^{-}(x_p)}{|\nabla \phi(x_p)|}\right)^{p-4}\frac{(u_p^{-})^3(x_p)}{p-2}.
$$
Passing the limit as $p \to \infty$ in the above inequality we conclude that
$$
- \Delta_{\infty} \phi(x_0) \geq 0.
$$
That proves that $u_{\infty}$ is a viscosity super-solution.

Now, we will analyze the another case. To this end, fix $x_0 \in \{u_{\infty}<0\} \cap \Omega$ and a test function $\phi \in C^2(\Omega)$ such that $u_{\infty}(x_0) = \phi(x_0)$ and the inequality $u_{\infty}(x) < \phi(x)$ holds for $x \neq x_0$.
We want to prove that
\begin{equation}\label{eq5.2}
  - \Delta_{\infty} \phi(x_0) \leq 0 \quad \text{and} \quad -|\nabla \phi(x_0)|+\beta \phi^{-}(x_0) \leq 0.
\end{equation}
Again, as before, there exists a sequence $x_p \to x_0$ such that $u_p-\phi$ has a local maximum at $x_p$ and since $u_p$ is a weak sub-solution (resp. viscosity sub-solution) to \eqref{ecu}, we have that
$$
  -\frac{|\nabla \phi(x_p)|^2 \Delta \phi(x_p)}{p-2} - \Delta_{\infty} \phi(x_p) \leq  -\left( \frac{\beta_p^{\frac{1}{p-4}}u_p^{-}(x_p)}{|\nabla \phi(x_p)|}\right)^{p-4}\frac{(u_p^{-})^3(x_p)}{p-2} \leq 0.
$$
Thus, we obtain $- \Delta_{\infty} \phi(x_0) \leq 0$ letting $p \to \infty$. If $-|\nabla \phi(x_0)|+\beta_{\infty} \phi^{-}(x_0) > 0$, as $p \to \infty$, then the right hand side goes to $-\infty$, which clearly yields a contradiction  because $\phi \in C^2(\Omega)$. Therefore \eqref{eq5.2} holds.

The last part of the proof consists in proving that $u_{\infty}$ is a viscosity solution to
$$
  \min\{- \Delta_{\infty} u_{\infty}(x), |\nabla u_{\infty}(x)|-\alpha_{\infty} u^{+}_{\infty}(x)\} = 0 \quad \text{in} \quad \{u_{\infty}>0\} \cap \Omega.
$$
The argument is similar to the previous case and for this reason we will omit it.
\end{proof}

\section{Characterization of $\Sigma_\infty$: Proof of Theorems \ref{teo.1d} and \ref{teo.rn}}\label{sec4}

\subsection{The one-dimensional case} \label{sect-1-d}

As we pointed out in the introduction, the spectrum of \eqref{ecu} as $p\to\infty$ is completely understood when $n=1$ since, in this case, the structure of $\Sigma_p$  is explicitly determined. When $\Omega=(0,1)$,  $\Sigma_p$ it is composed by the two trivial lines
$$
\C_{1,p}^+= \R\times \{\lam_{1,p}\}, \qquad \C_{1,p}^-=  \{\lam_{1,p}\} \times \R,
$$
and the family of hyperbolic-like curves
$$
\C_{k,p}\, : \, \alpha_p^{-1/p}+\beta_p^{-1/p} = \frac{2}{k\pi_p }
$$
when $k$ is even, and
\begin{align*}
&\C_{k,p}^+\, : \, \frac{k-1}{2}\alpha_p^{-1/p}+\frac{k+1}{2}\beta_p^{-1/p} = \frac{1}{\pi_p }, \\
&\C_{k,p}^-\, : \, \frac{k+1}{2}\alpha_p^{-1/p}+\frac{k-1}{2}\beta_p^{-1/p} = \frac{1}{\pi_p }
\end{align*}
when $k$ is odd. Here $\pi_p$ is given by
$$
 \pi_p =  2(p-1)^{1/p} \int_0^1 \frac{ds}{(1-s^p)^{1/p}}.
$$

Observe that, since the eigenvalues of \eqref{plap} are explicitly given by $\lam_{k,p}=(k\pi_p)^p$, $k\in\N,$  the curves $\C_k^\pm$ can be rewritten in terms of them as
\begin{align*}
	&\C_{k,p}\, : \,   \left( \frac{\lam_{k,p}}{\alpha_p} \right)^\frac1p + \left( \frac{\lam_{k,p}}{\beta_p} \right)^\frac1p=2 \qquad &\mbox{ if } k \mbox{ is even} \\\
	&\C_{k,p}^+\, : \, \left( \frac{\lam_{(k-1)/2,p}}{\alpha_p} \right)^\frac1p + \left( \frac{\lam_{(k+1)/2,p}}{\beta_p} \right)^\frac1p=1\qquad &\mbox{ if } k \mbox{ is odd}\\
	&\C_{k,p}^-\, : \, \left( \frac{\lam_{(k+1)/2,p}}{\alpha_p} \right)^\frac1p + \left( \frac{\lam_{(k-1)/2,p}}{\beta_p} \right)^\frac1p=1.
\end{align*}

\begin{proof}[Proof of Theorem \ref{teo.1d}]

In view of \eqref{lam2}, the trivial lines $\C_{1,p}^\pm$ converge to
$$
	\C_{1,\infty}^+=\R\times \{ 2 \}, \qquad \C_{1,\infty}^-= \{2\}\times \R.
$$
as $p\to\infty$, since $\pi_p\to 2$ as $p\to \infty$.

Let us analyze the nontrivial curves $\C_{k,p}^+$ as $p\to\infty$ for $k$ odd.  According to the previous expressions, this hyperbolic curve  can be parametrized as
\begin{align*}
\mathcal{C}_{k,p}^+=\{ (\alpha_p(s),\beta_p(s)), \, s\in \R^+\}
\end{align*}
where
$$
\alpha(s)=\left(\lam_{\frac{k-1}{2},p}^\frac1p + s^{-1} \lam_{\frac{k+1}{2},p}^\frac1p\right)^p, \qquad \beta(s)=s^p \alpha(s).
$$
Here $(\alpha_p(s),\beta_p(s))$ denotes the intersection between $\C_{k,p}$ and the line of slope $s^p$
passing through the origin in $\R^2$. Observe that when $s=1$, it follows that  $\alpha_p=\beta_p=(k\pi_p)^p=\lam_{k,p}$.

If we define the curve
$$
	\C_{k,\infty}^+ \defeq \{(\alpha_\infty(s),\beta_\infty(s)), \, s\in \R^+\}
$$
where
$$
\alpha_\infty(s) \defeq \lim_{p\to\infty} \alpha_p(s)^{1/p}, \qquad \beta_\infty(s)\defeq \lim_{p\to\infty} \beta_p(s)^{1/p},
$$
from \eqref{lam1}, we get
\begin{align*}
\alpha_\infty(s)=k-1+s^{-1}(k+1), \qquad
\beta_\infty(s)= k+1+s(k-1).
\end{align*}

Observe that, from \eqref{lamk.1d}, we have that $\lam_{k,\infty}=2k$. In particular, when $s=1$, the curve $\C_{k,\infty}^+$ passes through the point $(2k,2k)$, as expected.

The previous expressions lead to the formula for $\C_{k,\infty}^+$ in the case in which $k$  is odd. In a similar way can be obtain the formulas for the remaining curves, and the proof is complete.
\end{proof}

\subsection{The general case}\label{GenCase}

As it was described in the introduction, one immediately observe that $\Sigma_p$ contains the trivial lines $\{\lam_{1,p}\}\times \R$ and $\R \times \{\lam_{1,p}\}$, being $\lam_{1,p}$  the first eigenvalue given by \eqref{eqEigen}.  However, in contrast with the one-dimensional case, where a full description of the spectrum is available, when $n>1$ it is only known the existence of a curve $\C_{2,p}$ beyond the trivial lines. Such a curve has an hyperbolic shape and it is proved to be variational (see for instance \cite{dF-Go-94}). As far as we are concerned, we will consider the characterization given in \cite{terra}, in which the intersection of $\C_{2,p}$ with the line of slope $t\in \R^+$ passing through the origin in $\R^2$ can be written as
\begin{equation} \label{par.ab}
 	 (\alpha_p(t),\beta_p(t))=(t^{-1} c_p(t),c_p(t)),\quad  t\in\R^+
\end{equation}
where
$$
 c_p(t)=\inf_{\PP_2} \max\{t \lam_{1,p}(\omega_1), \lam_{1,p}(\omega_2)\}, \qquad t\in \R^+
$$
being $\PP_2=(\omega_1, \omega_2)$  the class of partitions in two disjoint, connected, open subsets of $\Omega$, and $\lam_{1, p}(\omega)$ the first eigenvalue of \eqref{plap} in $\Omega=\omega$.

\begin{proof} [Proof of Theorem \ref{teo.rn}]
The proof follows taking limit as $p\to\infty$ in the characterization of the curves $\C_{1,p}^\pm$ and $\C_{2,p}$.

The expression for $\C_{1,\infty}^\pm$ follows from \eqref{lam1}.
Now, if we define the function
$$
c_\infty(t)\defeq \lim_{p\to\infty} \inf_{\PP_2} \max\Big\{t \lam_{1,p}(\omega_1)^\frac1p, \lam_{1,p}(\omega_2)^\frac1p\Big\}, \qquad t\in \R^+
$$
again, from \eqref{lam1}, the following characterization holds
\begin{equation} \label{c.infty}
c_\infty(t)=\inf_{\PP_2} \max\left\{\frac{t}{r_1}, \frac{1}{r_2}\right\}, \qquad t\in \R^+
\end{equation}
where, given $(\omega_1, \omega_2)\in\PP_2(\Omega)$,  $r_i$ is the radius of the biggest ball contained in $\omega_i$, for $i=1,2$.

For $t\in \R^+$, defining the functions
\begin{equation} \label{curva2}
\alpha_\infty(t)=t^{-1} c_\infty(t), \quad \beta_\infty(t)=c_\infty(t),
\end{equation}
from \eqref{par.ab} it follows the desired parametrization of $\C_{2,\infty}$.

Now, let us see that, in fact, there is no point in $\Sigma_\infty$ between the trivial lines and $\C_{2,\infty}$, i.e., the first nontrivial curve is lower isolated when it detaches from the trivial lines. Suppose otherwise, that there is some  $(\alpha_0,\beta_0)$ strictly between these curves, and denote $u$ the corresponding eigenfunction. Observe that $u$ cannot have constant sign in $\Omega$, since this would imply that $(\alpha_0,\beta_0)\in \C_{1,\infty}^\pm$. Then, there exists a nontrivial partition $(p_1,p_2)\in \PP_2(\Omega)$ such that $u>0$ in $p_1$ and $u<0$ in $p_2$. Now, if we consider $t_0=\frac{\rho_1}{\rho_2}$, it is clear that the inequalities
$$
	\beta(t_0)> \beta_0(t_0), \qquad \alpha(t_0)> \alpha_0(t_0)
$$
are strict, being $(\alpha(t_0),\beta(t_0))\in \C_{2,\infty}$  and $\rho_i$ the radius of the biggest ball inside $p_i$, $i=1,2$. However, by the definition of the first nontrivial curve \eqref{curva2}, it must hold that
$$
\beta(t_0) = c_\infty(t_0)=\inf_{\PP_2} \max\left\{ \frac{t_0}{r_1}, \frac{1}{r_2}\right\} =  \max\left\{ \frac{t_0}{\rho_1} , \frac{1}{\rho_2}\right\} = \beta_0(t_0),
$$
a contradiction. Consequently, $\C_{2,\infty}$ is lower isolated when it is different from the trivial lines.

Finally, we observe the following: if we take a ball of radius $$\mathfrak{r}(\Omega)=\max\limits_{x\in\Omega} \dist(x,\partial \Omega)$$ inside $\Omega$ and there is some room left, that is, $\Omega \setminus \overline{B_{\mathfrak{r}}}
\not=\emptyset$, then we can consider as a partition of $\Omega$ the sets $\omega_1 = B_{\mathfrak{r}}$ and
$\omega_2 = \Omega \setminus \overline{B_{\mathfrak{r}}}$. From our previous arguments we get that the points
$(1/\mathfrak{r}(\Omega) , 1/ \mathfrak{r} (\Omega \setminus \overline{B_{\mathfrak{r}}}))$ and $( 1/ \mathfrak{r} (\Omega \setminus \overline{B_{\mathfrak{r}}}), 1/\mathfrak{r}(\Omega) )$ belong to
$\C_{2,\infty} \cap \C_{1,p}^\pm$. Therefore, the curve $\C_{2,\infty}$ intersects the trivial lines
when $\Omega \setminus \overline{B_{\mathfrak{r}}}
\not=\emptyset$ and we just observe that this holds for every domain that is different from a ball.
\end{proof}
As a corollary, the following properties  are fulfilled by $\C_{2,\infty}$.

\begin{corollary}\label{cor4.1} The following statements hold true:
\begin{enumerate}
  \item[(a)] $\C_{2,\infty}$ is a continuous and non-increasing curve, symmetric with respect to the
diagonal.
  \item[(b)] $\C_{2,\infty} \subset \left\{(x, y) \in \R^2 \suchthat x, y \geq \lambda_{1, \infty}(\Omega) \right\} \setminus \left\{(x, y) \in \R^2 \suchthat x, y > \lambda_{2, \infty}(\Omega)\right\}$.
  \item[(c)] (Courant nodal domain theorem) Any eigenfunction associated to $(x, y)\in \C_{2,\infty} \setminus
  \C_{1,p}^\pm $ admits exactly two nodal domains.
\end{enumerate}
\end{corollary}

\begin{proof}
Symmetry of $\C_{2,\infty}$ arises from interchanging the roles of $\omega_1$ and $\omega_2$ in the the expression of $c_\infty(t)$. Continuity and monotonicity of $\C_{2,\infty}$ follow from the definition of $c_\infty(t)$.

The curve $\C_{2,\infty}$ always is above or coincides with the trivial lines since $\lam_{1,\infty}(\Omega)\leq c_\infty(\Omega)$. Furthermore, any point belonging to $\C_{2,\infty}$ does not belong to $\{x,y>\lam_{2,\infty}\}$. In fact,
since $\C_{2,\infty}$ is continuous, all path linking $(\lam_{2,\infty},\lam_{2,\infty})$ to any point in $\{x,y>\lam_{2,\infty}\}$ should increase at some moment, which would contradict the non-increasing nature of the curve.

Finally, by construction, any eigenfunction corresponding to a point of the curve admits exactly two nodal domains.
\end{proof}

\begin{remark}
{\rm Let $(\alpha, \beta)$ a point in $ \C_{2,\infty} \cap
  \C_{1,p}^\pm $, that is, for example a point of the form $(1/\mathfrak{r}(\Omega), \beta)$
  with $\beta$ large. For those points there are at least two different eigenfunctions (this
  point of the spectrum is not simple). In fact, there is a positive eigenfunction (that comes
  from the limit as $p\to \infty$ in the first eigenvalue for the $p-$Laplacian) and another one
  that changes sign (that can be obtained from our construction since we assumed that
  $(\alpha, \beta)\in \C_{2,\infty}$). Therefore, $\Sigma_\infty$ has eigenvalues with
  multiplicity on the trivial lines (this fact does not happen for $\Sigma_p$).}
\end{remark}

\section{Classifying the $\infty$-{F}u{\v{c}}ik spectrum}\label{ClasFucikSpect}

In this section we will study different families of domains based on the shape of the curve $\C_{2,\infty}$. As we will see, given a domain $\Omega\subset \R^n$, this classification will depend only on $\mathfrak{r}(\Omega)$, the radius of the biggest ball contained in $\Omega$, and $ \mathfrak{R}(\Omega)$, the maximum radius of a couple of balls of the same size fitted inside $\Omega$.

Regardless the configuration of $\Omega$, it always holds that the lines $y= \frac{1}{\mathfrak{r}(\Omega)}$ and $x=\frac{1}{\mathfrak{r}(\Omega)}$ define the trivial lines in the spectrum, and that the point $\left(\frac{1}{
\mathfrak{R}(\Omega)}, \frac{1}{\mathfrak{R}(\Omega)}\right)$ belongs to $\Sigma_\infty$. As we will see, the shape of $\C_{2,\infty}$ depends on the relation between $\mathfrak{r}(\Omega)$ and $\mathfrak{R}(\Omega)$.

Hereafter, given a partition $(\omega_1, \omega_2)\in \PP_2(\Omega)$ we will denote $r_1$ and $r_2$ the radii of the biggest balls contained in each component.

We will distinguish two classes of domain depending on whether the corresponding curve $\C_{2,\infty}$ intersects or not the trivial lines.

\begin{figure}[ht]
\begin{center}
\includegraphics[width=13.1cm,height=4.8cm]{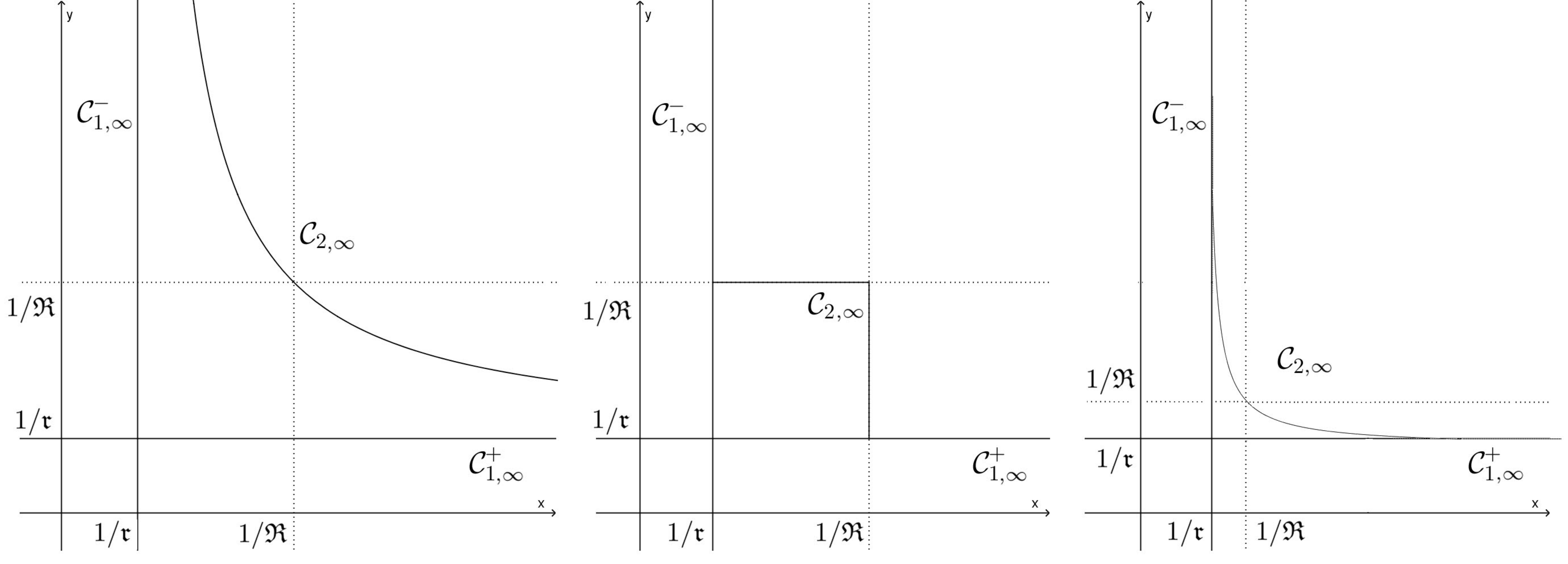}
\caption{The first three curves in $\Sigma_\infty$ for a domain type I (left) and type II (middle and right).}
\label{dib3}
\end{center}
\end{figure}

\subsection{Type I}
 Here lies the domain $\Omega$ whose curve $\C_{2,\infty}$ is hyperbolic and asymptotic to the trivial lines.
 As we have seen the only possibility is a ball.

\begin{example}[{\bf A ball}]
Let us consider the domain $\Omega$ given by a open ball of radius $R$ in $\R^n$. It is immediate that $\lam_{1,\infty}(\Omega)=\frac{1}{R}$.

Since the radii of two tangential balls fitted in $\Omega$ must satisfy $r_1+r_2=R$, the expression of \eqref{c.infty} will be minimized  will be minimized when $\frac{t}{r_1}=\frac{1}{r_2}$, i.e.,
$$
t=\frac{r_1}{R-r_1}, \qquad \text{ which implies } \quad  r_1=\frac{Rt}{1+t}.
$$
In such case it follows that
$$c_\infty(t)=\frac{1+t}{R}.$$

Finally, observe that values of $t$ approaching zero correspond to a partition $(\omega_1, \omega_2)\in \PP(\Omega)$ in which the biggest ball in $\omega_2$ is almost the whole $\Omega$ and the biggest one in $\omega_1$ is very small; values of $t$ approaching $+\infty$ correspond to a partition in which the balls interchange their roles: the biggest ball in $\omega_1$ is almost the whole $\Omega$. See figure \ref{dib1}.

Consequently, according equation \eqref{curva.2}, the curve $\C_{2,\infty}$ is given by
$$
    \C_{2,\infty}=\left\{ \left( \frac{1+t}{Rt}, \frac{1+t}{R}\right),\, t\in \R^+\right\}.
$$
Observe that when $t=1$ the curve contains the point $\left(\frac{2}{R}, \frac{2}{R}\right)$, which is precisely corresponds to $(\lam_{2,\infty}(\Omega), \lam_{2,\infty}(\Omega))$.

\begin{figure}[ht]
\begin{center}
\includegraphics[width=13.1cm,height=4.4cm]{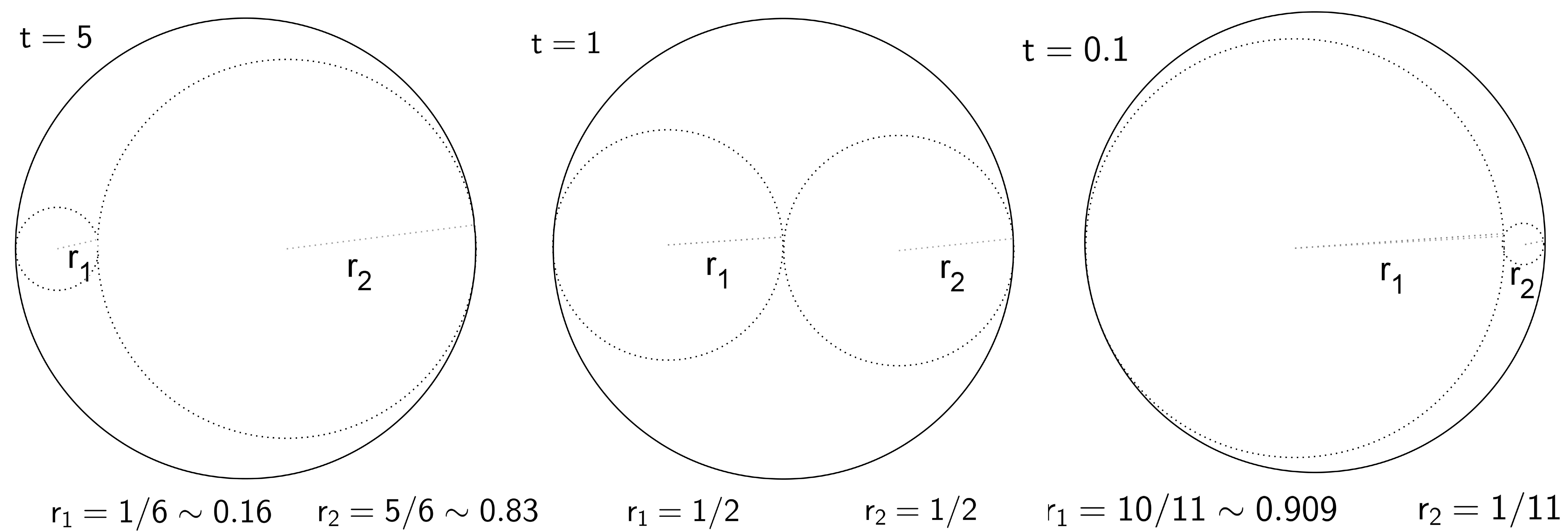}
\caption{Partitions corresponding to $t=1$ ($r_1\sim0.16$ and $r_2\sim 0.83$), $t=1$ ($r_1=r_2=0.5)$ and $t=0.1$ ($r_1\sim0.909$ and $r_2\sim 0.09$)} for $\Omega$ the unit ball .
 \label{dib1}
\end{center}
\end{figure}
\end{example}

\begin{example}[{\bf The unit interval}]
When $\Omega$ is considered to be an open interval in the real line, the picture of $\Sigma_\infty$ is analogous to $\Sigma_p$ for a fixed value of $p$, i.e., it consists in a sequence of hyperbolic-like curves, as it is showed in Figure \ref{dib0}.

Moreover, since explicit formulas for the eigenfunctions are known for a fixed value of $p$, it is possible to describe the profile of the limit problem. For instance, if $(\alpha,\beta) \in \C_{2,\infty}$, the corresponding eigenfunction $u_{\infty}$ will be given by
\begin{align*}
\lim_{p\to\infty} u_p(x) = u_{\infty}(x)=
\begin{cases}
x &\quad \text{in } \left(0,\frac{\ell}{2}\right]\\
-x+\ell &\quad \text{in } \left[\frac{\ell}{2},\frac{\ell+1}{2}\right]\\
x-1 &\quad \text{in } \left[\frac{\ell+1}{2}, 1\right)
\end{cases}
\end{align*}
where $\ell\in(0,1)$
and
\begin{align*}
u_p(x)=
\begin{cases}
\displaystyle \sin_p\left(\frac{\pi_p x}{\ell}\right) &\quad \text{in } (0,\ell]\\
\displaystyle -\sin_p\left(\frac{\pi_p x}{1-\ell}\right) &\quad \text{in } [\ell,1).
\end{cases}
\end{align*}
Finally, notice that $u_{\infty}$ is a viscosity solution to \eqref{ecu.infty} with $(\alpha(l), \beta(l)) = \left( \frac{2}{l}, \frac{2}{1-l}\right)$.

\begin{figure}[ht]
\begin{center}
\includegraphics[width=11cm,height=7.5cm]{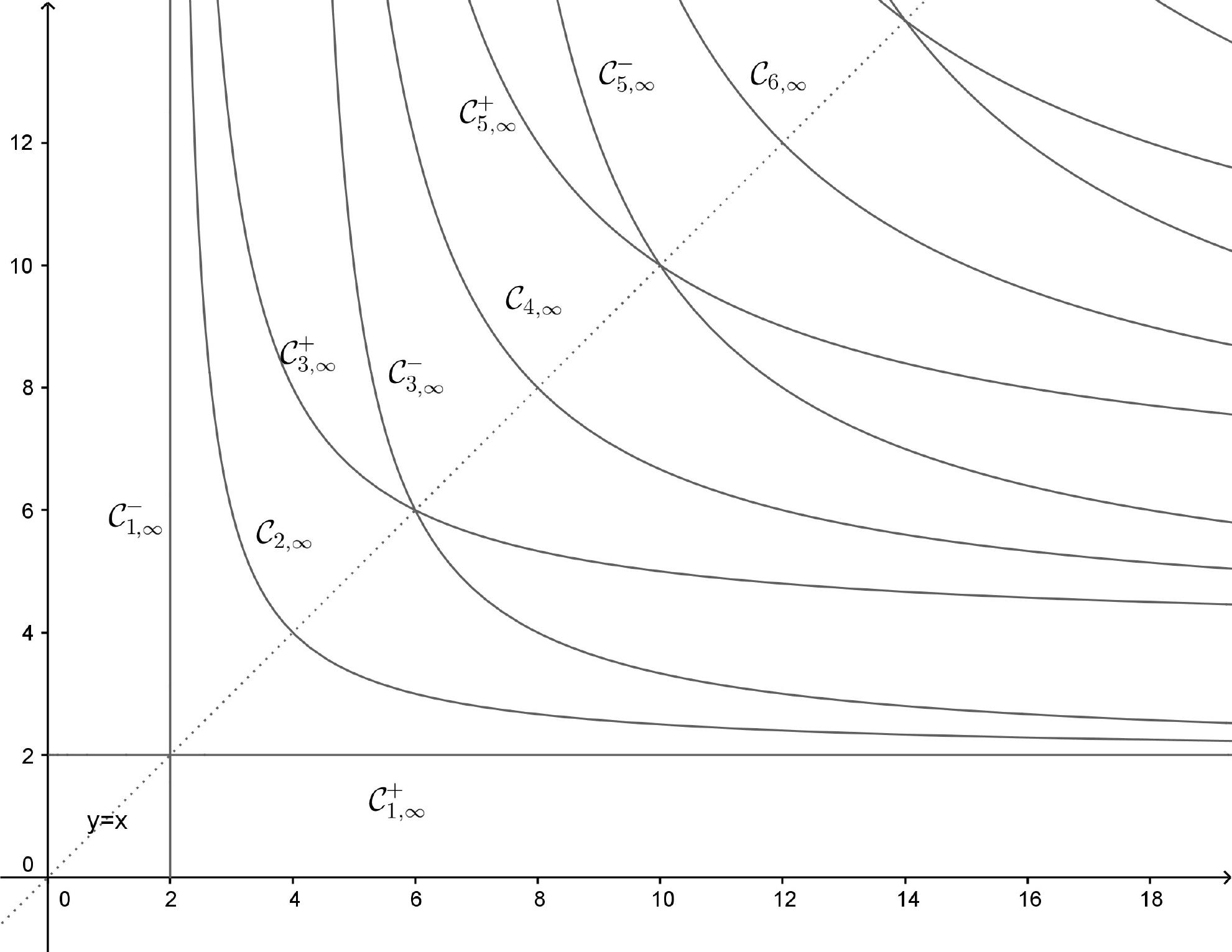}
\caption{ $\Sigma_\infty$ for the unit interval in $\R$.}
\label{dib0}
\end{center}
\end{figure}
\end{example}

\subsection{Type II}
 All domain $\Omega$ whose curve $\C_{2,\infty}$ intersects the trivial lines. We can also subdivide this category in $\C_{2,\infty}$ connecting the trivial lines (type II.A) and $\C_{2,\infty}$ totally contained in the trivial lines (type II.B).

\begin{example}[{\bf ``Linked'' balls}]
Given $R_1\leq R_2$, let us consider a domain $\Omega$  made as the union of the balls of radii $R_1$ and $R_2$ by means of a tube of length $\ve<R_1$.

Since the radius of the biggest ball contained in $\Omega$ is $R_2$, we get $\lam_{1,\infty}= \frac{1}{R_2}$. Now, the couple of biggest balls contained in $\Omega$ have radius $R_1$. If we fix $r$ such that $r=R_1$, the expression of $c_\infty$ will be minimized when $\frac{t}{r}=\frac{1}{R_1}$, that is, when both coefficients inside the maximum in the expression of $c_\infty$ are the same. In this case, $c_\infty=\frac{1}{R_1}$.

Observe that the case $r=R_1$, corresponds to $t=1$. As $r$ increases, the value of $t$ decreases. This process finishes when $r=R_2$, which corresponds to $t=\frac{R_2}{R_1}$.

Analogously, this process can be made interchanging the roles of $R_1$ and $R_2$, leading to the following expression
for the second non-trivial curve
$$
 \C_{2,\infty}=\left\{ \left( \frac{1}{tR_1}, \frac{1}{R_1}\right),\, 1 \leq t \leq \frac{R_2}{R_1} \right\} \cup \left\{ \left( \frac{1}{R_1}, \frac{1}{tR_1}\right),\, 1 \leq t \leq \frac{R_2}{R_1} \right\}.
$$

It is remarkable to see that in the extremal case $R_1=R_2$ the curve $\C_{2,\infty}$ is contained in the trivial curves.  This situation occurs when the radius of the biggest ball contained in $\Omega$ coincides with the radius of biggest couple of balls of the same size fitted in $\Omega$ (for example in an annular domain or more generally in a stadium domain).

It is also straightforward to see that the analysis made above only depends of radius of the biggest ball contained in $\Omega$ and the  radius of biggest couple of identical balls contained in $\Omega$. Consequently, the three domains exhibited in figure \ref{dib2} have the same first curves $\C_{1,\infty}^\pm$ and $\C_{2,\infty}$.

\begin{figure}[ht]
\begin{center}
\includegraphics[width=13cm,height=2.8cm]{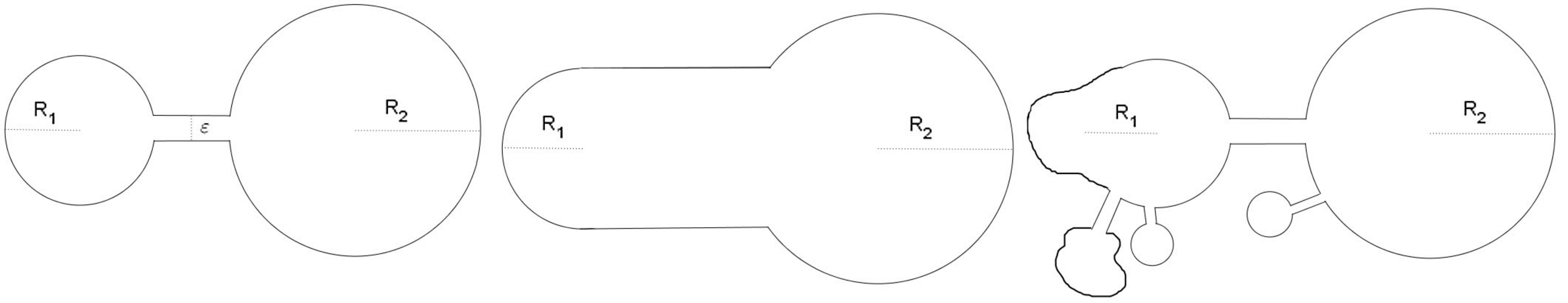}
\caption{Three domains for which the first three curves in $\Sigma_\infty$ coincide.}
\label{dib2}
\end{center}
\end{figure}
\end{example}

\begin{example}[{\bf The unit cube}] Let us consider $\Omega$ be the unit square in $\R^2$, $\Omega = (0,1)\times (0,1)$. In this case, since the biggest ball fitted in $\Omega$ has radius $R=\frac{1}{2}$ we have that $\lam_{1,\infty}(\Omega)=2$.

Let us analyze the second nontrivial curve. When we compute $c_\infty(t)$ we must consider two balls of radii $r_1$ and $r_2$ contained in $\Omega$ such that $\frac{t}{r_1}$ and $\frac{1}{r_2}$ coincide. Notice that for  $t=1$ we obtain
$$
  r_1=r_2=\frac{\sqrt{2}}{2(1+\sqrt{2})}.
$$
But we can also consider a partition such that $r_2$ increases and $r_1$ decreases; since both balls are fitted in $\Omega$, they must verify $$r_1+r_2=\frac{\sqrt{2}}{1+\sqrt{2}}.$$
In this case, if
$$
	t=\frac{1}{r_1}\frac{\sqrt{2}}{1+\sqrt{2}}-1
$$
we can guarantee that $\frac{t}{r_1}$ and $\frac{1}{r_2}$ coincide. Observe that the computations made to enlarge $r_1$ (and then to obtain a smaller $r_2$) with the previous expression of $t$ can be performed provided that $r_1\leq \frac{1}{2}$. This procedure gives that
\begin{equation} \label{tt1}
	c_\infty(t)=\frac{t}{r_1}=\frac{1}{r_2}=(t+1)\Big( 1+\frac{\sqrt{2}}{2} \Big), \qquad 	\frac{2\sqrt{2}}{1+\sqrt{2}}-1 \leq t \leq 1.
\end{equation}

Now we can fix the value $r_2=\frac{1}{2}$ to be the maximum radius of a ball fitted in $\Omega$ and to continue decreasing the value of $r_1$. In this case, when considering
$$
	t=2r_1
$$
we can assure that $\frac{t}{r_1}$ and $\frac{1}{r_2}$ coincides to be equal to $2$. This process can be continued as $r_2\to 0$ obtaining
\begin{equation} \label{tt2}
	c_\infty(t)= 2, \qquad 	0 \leq t \leq \frac{2\sqrt{2}}{1+\sqrt{2}}-1.
\end{equation}
From \eqref{tt1} and \eqref{tt2} we get that
$$
 \C_{2,\infty}^1= \left\{
 \begin{array}{ll}
 \displaystyle
 \left( \frac{2}{t}, 2\right) & \displaystyle\qquad 0 \leq t \leq \frac{2\sqrt{2}}{1+\sqrt{2}}-1\\
 \displaystyle \left( \frac{\tau(t+1)}{t}, \tau(t+1) \right) &
 \displaystyle \qquad \frac{2\sqrt{2}}{1+\sqrt{2}}-1 \leq t \leq 1
 \end{array} \right.
 $$
where $\tau=1+\frac{\sqrt{2}}{2}$.

Observe that this construction can be made, analogously, interchanging the roles of $r_1$ and $r_2$, leading to
$$
 \C_{2,\infty}^2=
 \left\{
 \begin{array}{ll}
 \displaystyle
 \left( 2, \frac{2}{t} \right) &\qquad \displaystyle 0 \leq t \leq \frac{2\sqrt{2}}{1+\sqrt{2}}-1\\
 \displaystyle \left( \tau(t+1), \frac{\tau(t+1)}{t}  \right) &
 \displaystyle \qquad \frac{2\sqrt{2}}{1+\sqrt{2}}-1 \leq t \leq 1,
 \end{array} \right.
$$
and consequently, $\C_{2,\infty}= \C_{2,\infty}^1 \cup \C_{2,\infty}^2$. See Figure \ref{dib5}.

\begin{figure}[ht]
\begin{center}
\includegraphics[width=13cm,height=4.2cm]{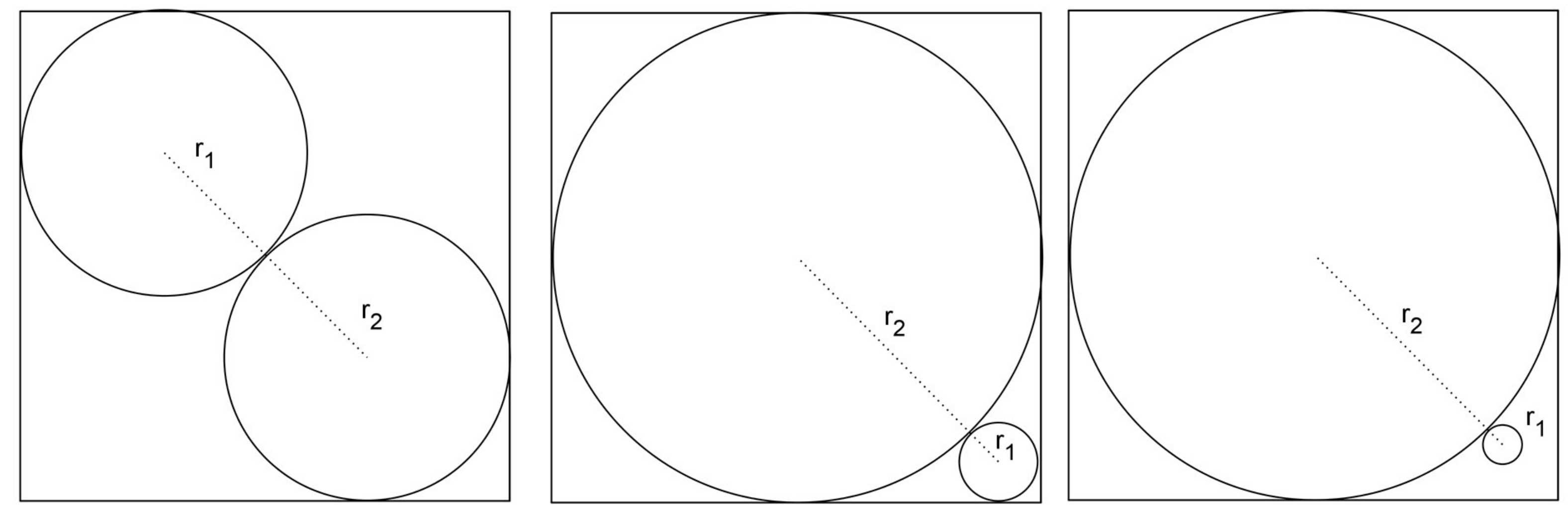}
\caption{ Partitions corresponding to $\Omega$ the unit cube in $\R^2$ for $t=1$, i.e., $r_1=r_2=\frac{\sqrt{2}}{2(1+\sqrt{2})}$ (left); $t=\frac{2\sqrt{2}}{1+\sqrt{2}}$, i.e., $r_2=\frac{1}{2}$, $r_1=\frac{\sqrt{2}}{1+\sqrt{2}}-\frac{1}{2}$ (middle) and $t=0.16$ with $r_2=\frac{2\sqrt{2}}{1+\sqrt{2}}$ and $r_1=0.08$ (right) }
\label{dib5}
\end{center}
\end{figure}

It is remarkable to see that for values of $t$ in the range $[0,\frac{2\sqrt{2}}{1+\sqrt{2}}-1]$, the curve $C_{2,\infty}$ is contained in the trivial lines, i.e., the first intersection among the second nontrivial curve and the trivial lines occurs at $(\tau_0,2)$ and $(2,\tau_0)$, where $\tau_0=\frac{2(\sqrt{2}+1)}{\sqrt{2}-1}$. See Figure \ref{dib4}.
\end{example}

\begin{figure}[ht]
\begin{center}
\includegraphics[width=7cm,height=5.3cm]{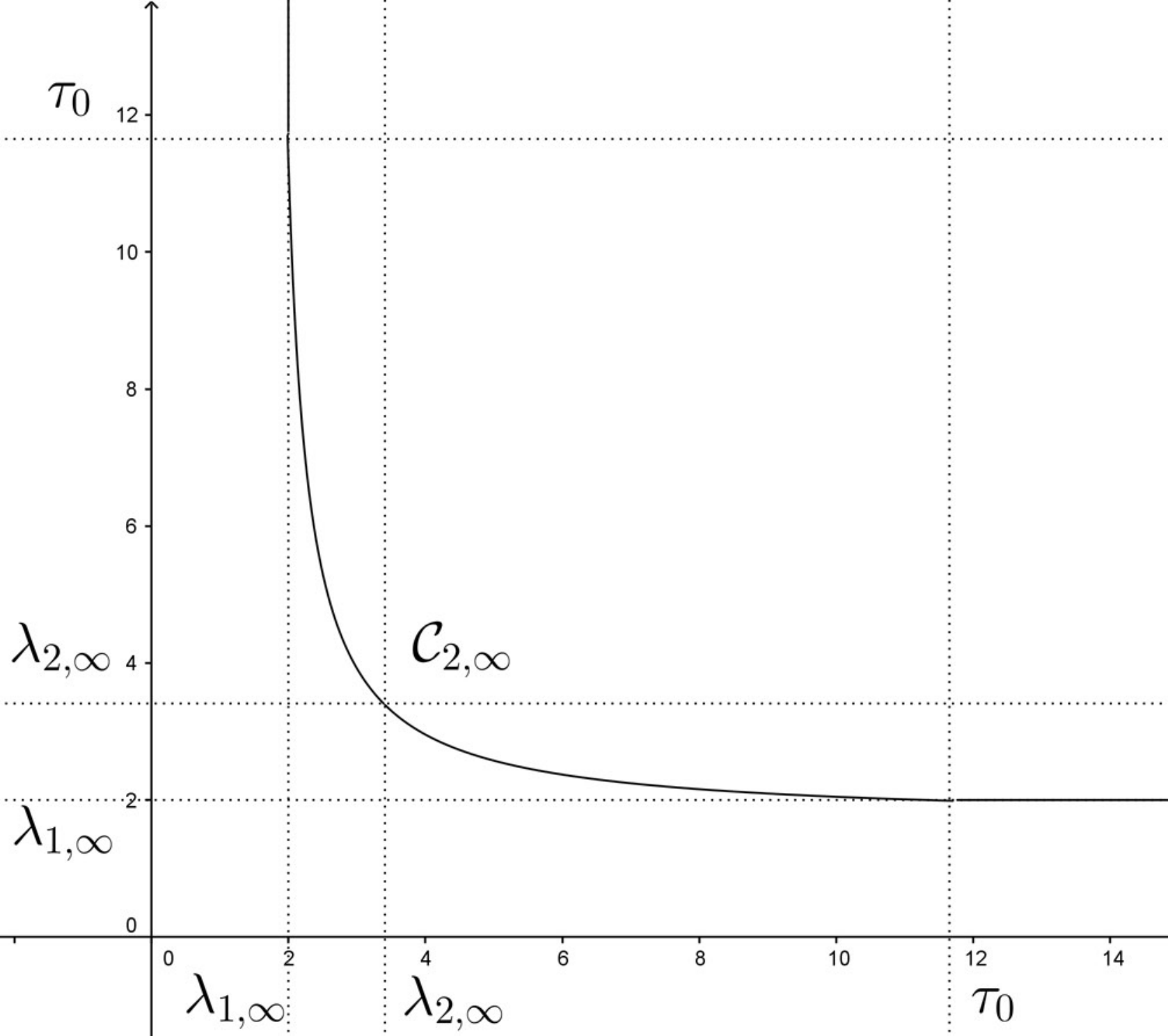}
\caption{ The curve $C_{2,\infty}$ for the unit cube in $\R^2$. }
\label{dib4}
\end{center}
\end{figure}

\vspace{2cm}

{\bf Acknowledgments}
This work has been partially supported by Consejo Nacional de Investigaciones Cient\'{i}ficas y T\'{e}cnicas (CONICET-Argentina). JVS would like to thank the Dept. of Math. FCEyN, Universidad de Buenos Aires for providing an excellent working environment and scientific atmosphere during his Postdoctoral program.

\bibliographystyle{amsplain}

\providecommand{\bysame}{\leavevmode\hbox to3em{\hrulefill}\thinspace}
\providecommand{\MR}{\relax\ifhmode\unskip\space\fi MR }
\providecommand{\MRhref}[2]{%
  \href{http://www.ams.org/mathscinet-getitem?mr=#1}{#2}
}
\providecommand{\href}[2]{#2}

\end{document}